\documentclass{amsart}

\usepackage[T1]{fontenc}
\usepackage[letterpaper,margin=1in]{geometry} 
\usepackage{mathtools, verbatim, 
  rotating, comment, epic, amsmath, amscd, amsfonts, amsthm, amssymb, enumerate, graphicx, cases, tablefootnote}
\usepackage[shortlabels]{enumitem}
\usepackage{url}
\usepackage{hyperref}
\usepackage{multirow, longtable}
\usepackage[matrix,arrow,curve]{xy}

\newtheorem{thm}{Theorem}[section]
\newtheorem{lem}[thm]{Lemma}

\newtheorem{prop}[thm]{Proposition}
\newtheorem{cor}[thm]{Corollary}

\newtheorem{qn}[thm]{Question}

\theoremstyle{definition}

\newtheorem{defn}[thm]{Definition}

\theoremstyle{remark}

\DeclareMathOperator{\disc}{disc}

\DeclareMathOperator{\Cl}{Cl}

\DeclareMathOperator{\Sym}{Sym}

\DeclareMathOperator{\Mat}{Mat}

\newcommand{\GL}{\mathrm{GL}}

\newcommand{\SL}{\mathrm{SL}}

\newcommand{\FF}{\mathbb{F}}

\newcommand{\QQ}{\mathbb{Q}}
\newcommand{\PP}{\mathbb{P}}
\newcommand{\RR}{\mathbb{R}}

\newcommand{\VV}{\mathbb{V}}
\newcommand{\ZZ}{\mathbb{Z}}

\newcommand{\C}{\mathcal{C}}

\newcommand{\E}{\mathcal{E}}

\newcommand{\G}{\mathcal{G}}

\renewcommand{\L}{\mathcal{L}}
\newcommand{\M}{\mathcal{M}}

\newcommand{\R}{\mathcal{R}}
\renewcommand{\S}{\mathcal{S}}
\newcommand{\T}{\mathcal{T}}

\newcommand{\V}{\mathcal{V}}
\newcommand{\WW}{\mathcal{W}}

\newcommand{\ba}{\overline}
\newcommand{\bs}{\backslash}

\newcommand{\tensor}{\otimes}

\renewcommand{\to}{\mathop{\rightarrow}\limits}

\newcommand{\size}[1]{\lvert #1 \rvert}
\newcommand{\Size}[1]{\left\lvert #1 \right\rvert}

\newcommand{\intsec}{\cap}
\newcommand{\union}{\cup}

\newcommand{\isom}{\cong}
\newcommand{\<}{\left\langle}
\renewcommand{\>}{\right\rangle}
\renewcommand{\(}{\left(}
\renewcommand{\)}{\right)}

\newcommand{\ignore}[1]{}

\renewcommand{\epsilon}{\varepsilon}

\newcommand{\mat}[4]{%
  \begin{bmatrix}
    #1 & #2 \\
    #3 & #4
  \end{bmatrix}
}

\newcommand{\vect}[2]{%
  \begin{bmatrix}
    #1 \\
    #2
  \end{bmatrix}
}

\newcommand{\smallvect}[2]{%
  \left[\begin{smallmatrix}
    #1 \\
    #2
  \end{smallmatrix}\right]
}
\newcommand{\tall}{\vphantom{\ensuremath{^{2^2}}}}
\numberwithin{equation}{section}

\begin{document}
\title{Diophantine approximation on conics}
\author{Evan M. O'Dorney}
\address{Department of Mathematical Sciences, Carnegie Mellon University, Pittsburgh, PA}


\date{September 20, 2022}


\begin{abstract} 
  Given a conic $\mathcal{C}$ over $\mathbb{Q}$, it is natural to ask what real points on $\mathcal{C}$ are most difficult to approximate by rational points of low height. For the analogous problem on the real line (for which the least approximable number is the golden ratio, by Hurwitz's theorem), the approximabilities comprise the classically studied Lagrange and Markoff spectra, but work by Cha--Kim and Cha--Chapman--Gelb--Weiss shows that the spectra of conics can vary. We provide notions of approximability, Lagrange spectrum, and Markoff spectrum valid for a general $\mathcal{C}$ and prove that their behavior is exhausted by the special family of conics $\mathcal{C}_n : XZ = nY^2$, which has symmetry by the modular group $\Gamma_0(n)$ and whose Markoff spectrum was studied in a different guise by  A.~Schmidt and Vulakh. The proof proceeds by using the Gross-Lucianovic bijection to relate a conic to a quaternionic subring of $\operatorname{Mat}^{2\times 2}(\mathbb{Z})$ and classifying invariant lattices in its $2$-dimensional representation.
\end{abstract}

\maketitle

\section{Introduction}

Let
\begin{equation}\label{eq:conic}
  \Phi(X,Y,Z) = AX^2 + BXY + CY^2 + DXZ + EYZ + FZ^2
\end{equation}
be a ternary quadratic form whose coefficients $A,\ldots, F$ are integers. Suppose that $\Phi$ is \emph{isotropic} over $\QQ$, that is, there is a vector $P \in \QQ^3\bs \{0\}$ with $\Phi(P) = 0$, or equivalently a point $[P] \in \C(\QQ)$ on the associated conic $\C = \C_{\Phi} = \VV(\Phi = 0)$. (Throughout this paper, if $P = (x_1,\ldots, x_n)$ is a nonzero point in affine space, we denote by $[P] = [x_1 : \ldots : x_n]$ the corresponding point in projective space.) A natural question is the following:
\begin{qn}
  How difficult is it to approximate a general point $[\Xi] \in \C(\RR)$ by rational points $[P] \in \C(\QQ)$ of small height? What points on $\C$ are the hardest to approximate in this way?
\end{qn}
Over $\QQ$, all such conics $\C$ are projectively equivalent; however, over the integers $\ZZ$, the behavior of the heights of rational points can vary significantly from one conic to another.

The \emph{discriminant} of $\Phi$ is four times the determinant of the associated matrix,
\begin{equation}
  \Delta_\Phi = 4 \det \begin{bmatrix}
    A & B/2 & D/2 \\
    B/2 & C & E/2 \\
    D/2 & E/2 & F
  \end{bmatrix}.
\end{equation}
Nonsingularity is equivalent to $\Delta_\Phi \neq 0$. By flipping the sign of the equation, we may assume $\Delta_\Phi > 0$, so $\Phi$ has signature $(1,2)$ as a quadratic form. The quadratic form $\Phi$ determines a bilinear form
\[
  \<P, Q\>_\Phi = \frac{\Phi(P + Q) - \Phi(P) - \Phi(Q)}{2}.
\]
We make the following definitions.
\begin{defn} Let $\Phi$ be an isotropic integer ternary quadratic form and $\C$ its associated conic.
  \begin{enumerate}[$($a$)$]
    \item If $[\Sigma] \in \PP^2(\RR)$ is a point with $\Phi(\Sigma) < 0$, let $[\Xi], [\Xi']$ be the points on $\C$ whose tangents intersect at $[\Sigma]$; they are given by intersecting $\C$ with the line $\<\Sigma, -\> = 0$. We define the \emph{approximability} of the pair $(\Xi, \Xi')$ to be the supremum
    \begin{equation} \label{eq:mu_Sigma}
      \mu([\Xi], [\Xi']) = \sup_{P \in \C(\ZZ)} \frac{\sqrt{\Delta_\Phi \size{\Phi(\Sigma)}}}{\size{\<\Sigma, P\>_\Phi}} \in \RR \union \{\infty\},
    \end{equation}
    and we define the \emph{Markoff spectrum} $\M(\C)$ of $\C$ to be the set of real values attained by $\mu([\Xi], [\Xi'])$ as  $[\Sigma]$ ranges over points with $\Phi(\Sigma) < 0$, or equivalently as $[\Xi], [\Xi']$ range over pairs of distinct points on $\C$.
    \item If $[\Xi] \in \C(\RR)$ is a point lying on $\C$, then let $[\Sigma] \neq [\Xi]$ be a point on the tangent line to $\C$ at $\Xi$; in other words,
    \[
      \<\Sigma, \Xi\> = 0.
    \]
    (We will prove in Lemma \ref{lem:well-def} that the choice of $\Sigma$ is immaterial.) We define the \emph{approximability} of $[\Xi]$ to be the limit supremum
    \begin{equation} \label{eq:lambda_Xi}
        \lambda\([\Xi]\) = \limsup_{P \in \C(\ZZ), [P] \to [\Xi]}\frac{\sqrt{\Delta_\Phi \size{\Phi(\Sigma)}}}{\size{\<\Sigma, P\>_\Phi}} \in \RR \union \{\infty\},
    \end{equation}
    and we define the \emph{Lagrange spectrum} $\L(\C)$ of $\C$ to be the set of real values attained by $\lambda(\Xi)$.
  \end{enumerate}
Strictly speaking, we should use the notations $\M(\C)$, $\L(\C)$ only when $\Phi$ is primitive, as scaling the equation defining $\C$ will scale the spectra likewise.
\end{defn}
The following is immediate from the definitions:
\begin{prop}~\label{prop:trans}
\begin{enumerate}[$($a$)$]
  \item\label{trans:pt} If $\Xi, \Xi' \in \C(\RR)$ are related by an $\SL_3(\ZZ)$-symmetry of $\C$, then $\lambda(\Xi) = \lambda(\Xi')$.
  \item\label{trans:conic} If two conics $\C$ and $\C'$ are related by an $\SL_3(\ZZ)$-transformation, they have the same Lagrange and Markoff spectra.
\end{enumerate}
\end{prop}
Thus we need only consider the finitely many $\SL_3(\ZZ)$-classes of conics with each discriminant. We will devote this paper to understanding the Lagrange and Markoff spectra of various conics.

\subsection{Previous work}\label{sec:prev}
The simplest conic, having discriminant $1$, is the Veronese curve
\[
  \C_1 : X Z - Y^2 = 0.
\]
A real point $\Sigma = (c, -\frac{b}{2}, a)$ gives a real indefinite quadratic form
\[
  f(s,t) = a s^2 + b s t + c t^2.
\]
Since the rational points on $\C_1$ are given in lowest terms by $[s^2 : s t : t^2]$ for $s, t$ coprime integers, we have
\[
  \mu\([\Xi], [\Xi']\) = \sup_{s,t \in \ZZ^2 \bs \{0\}} \frac{\sqrt{4ac - b^2}}{\size{as^2 + bst + ct^2}} = \sup_{s,t \in \ZZ^2 \bs \{0\}} \frac{\sqrt{\disc f}}{f(s,t)},
\]
so the Markoff spectrum of $\C$ is none other than the classical \emph{Markoff spectrum} $\M$ (see Cusick and Flahive \cite[p.~1]{CusickFlahive}). For ease of comparison, we may note that if $f(s,1)$ has roots $\xi, \xi' \in \RR$, then the associated tangency points are $[\Xi] = [\xi^2 : \xi : 1]$ and $[\Xi'] = [\xi'^2 : \xi' : 1]$, and
\begin{equation} \label{eq:mu_xi_xi'}
  \mu\([\Xi], [\Xi']\) = \sup_{s,t \in \ZZ^2 \bs \{0\}} \frac{\size{\xi - \xi'}}{t^2\Size{\dfrac{s}{t} - \xi}\Size{\dfrac{s}{t} - \xi'}} \eqqcolon \mu(\xi,\xi').
\end{equation}

Likewise, if $[\Xi] = [\xi^2 : \xi : 1]$ is an irrational point on $\C$, then, taking $\Sigma = [2\xi : 1 : 0]$, we get
\begin{equation} \label{eq:lambda_xi}
  \lambda\([\Xi]\) = \limsup_{\frac{s}{t} \in \QQ, \frac{s}{t} \to \xi} \frac{1}{\xi t^2 - st} = \limsup_{\frac{s}{t} \in \QQ, \frac{s}{t} \to \xi} \frac{1}{t^2 \Size{\dfrac{s}{t} - \xi}} \eqqcolon \lambda(\xi),
\end{equation}
so the Lagrange spectrum of $\C$ is none other than the classical \emph{Lagrange spectrum} $\L$ (\cite[p.~1]{CusickFlahive}).

The spectra $\L \subsetneq \M$ are closed subsets of $\RR$ with the following well-known properties \cite{CusickFlahive}:
\begin{itemize}
  \item On the interval $[\sqrt{5}, 3)$, $\L$ and $\M$ are discrete and equal, with all the least approximable points being quadratic irrationals parametrized by the binary tree of integer solutions to the \emph{Markoff equation} $x^2 + y^2 + z^2 = 3xyz$ \cite{MarkoffI,MarkoffII}.
  \item On the interval $[3, 4.5278\ldots)$, $\L$ and $\M$ are totally disconnected and difficult to describe fully. They are unequal (\cite{Freiman1968}; \cite{CusickFlahive}, ch.~3). They remain a topic of current research \cite{Erazo22}.
  \item Both $\L$ and $\M$ contain the infinite ray $[4.5278\ldots, \infty)$, known as \emph{Hall's ray.} The precise value of the lower endpoint
  \[
    4.5278\ldots = \frac{2221564096 + 283748\sqrt{462}}{491993569}
  \]
  was determined by Freiman in 1973 (\cite{FreimanHall}; \cite{Freiman1975}; \cite[ch.~4]{CusickFlahive}).
\end{itemize}

In \cite{ChaIntrinsic}, Cha and Kim consider the unit circle
\[
  \C : -X^2 - Y^2 + Z^2 = 0,
\]
a conic of discriminant $4$, finding that the initial discrete segment of $\L(\C)$ has a similar binary tree structure parametrized by the solutions of another Markoff-like equation $2x^2 + y^2_1 + y^2_2 = 4xy_1y_2$. This equation appeared earlier in the work of A.~Schmidt \cite{SchMinI} on what we will call the $2$-Markoff spectrum, which is more directly related to the conic
\[
  \C_2 : XZ = 2Y^2
\]
of discriminant $2$. The connection of Schmidt's work to approximation on the unit circle was apparently first noted by Kopetzky \cite{KopetzkyApprox}. In \cite{KimSim}, Kim and Sim study $\L(\C)$ and $\M(\C)$ further, in particular finding a Hall's ray. In \cite{ChaEis}, Cha, Chapman, Gelb, and Weiss consider the unique conic of discriminant $3$,
\[
  \C_3 : -X^2 - XY - Y^2 + Z^2 = 0 \sim XZ = 3Y^2,
\]
motivated by approximation of Eisenstein integers. (Here and henceforth, the relation $\sim$ denotes equivalence under a transformation in $\SL_3 \ZZ$.) The badly approximable numbers form a notably simpler structure, which we might describe as a ``rod'' (or a ``unary tree''): a single sequence, governed by a $2$-term linear recurrence, whose approximabilities converge linearly to the first accumulation point. In this paper, we study the question of Diophantine approximation on a general conic $\C$.

A few remarks are in order as regards normalization. When approximating a real point $\Xi = [\xi : \eta : 1]$ on the unit circle $\xi^2 + \eta^2 = 1$, we can take $\Sigma = [\eta : -\xi : 0]$ and get
\[
\lambda([\Xi]) = \limsup_{\substack{[a:b:c] \in \C(\ZZ),\\ [a:b:c] \to [\Xi]}} \frac{2c}{\Size{a \eta - b\xi}} = \limsup_{\substack{[a:b:c] \in \C(\ZZ),\\ [a:b:c] \to [\Xi]}} \frac{2}{\Size{\dfrac{a}{c} \eta - \dfrac{b}{c}\xi}}.
\]
Since $|a\eta - b\xi|$ is the Euclidean distance from the Pythagorean lattice point $(a,b)$ to the desired ray $\RR(\xi,\eta)$, our $\lambda$-values are $2$ times larger than those in \cite{ChaIntrinsic} and also $\sqrt{2}$ times larger than those in \cite{KimSim}, which studies $\M_{2}$ and $\L_{2}$ from yet another perspective derived from Hecke groups. Our choice of normalization is governed by the following two considerations:
\begin{itemize}
    \item If $\Xi$ is defined over a quadratic number field $\QQ[\sqrt{D}]$, the approximability $\lambda([\Xi])$ will also be defined over $\QQ[\sqrt{D}]$;
    \item $\L(\C) \subseteq \L$, $\M(\C) \subseteq \M$ will be contained in their classical analogues.
\end{itemize}

\subsection{Main result}

\begin{thm}\label{thm:main}
  Let $\C = \VV(\Phi)$ be a conic over $\QQ$ having a $\QQ$-point. There exist integers $n$ and $m$, with $nm^2 | \Delta_{\Phi}$, and a parametrization $\C \isom \PP^1$ such that, for any irrational points $[\Xi], [\Xi'] \in \C(\RR)$ with corresponding coordinates $\xi, \xi'$,
  \begin{alignat}{2}
    \lambda\([\Xi]\) &= m \cdot \max_{d|n} \lambda(d \xi) &&= m \cdot \limsup_{\frac{s}{t} \in \QQ, \frac{s}{t} \to \xi} \frac{\gcd(t, n)}{t^2 \Size{\dfrac{s}{t} - \xi}} \label{eq:lambda_main} \\
    \mu\([\Xi],[\Xi']\) &= m \cdot \max_{d|n} \mu(d \xi, d \xi') &&= m \cdot \sup_{s,t \in \ZZ^2 \bs \{0\}} \frac{\size{\xi - \xi'} \cdot \gcd(t, n)}{t^2\Size{\dfrac{s}{t} - \xi}\Size{\dfrac{s}{t} - \xi'}} \label{eq:mu_main}.
  \end{alignat}
\end{thm}
In words, the Lagrange and Markoff spectra are scalings by $m$ of the \emph{$n$-spectra,} which depend on $n$ alone and arise naturally from considering the conic $XZ = nY^2$. In general, any pair $(m,n)$ can be gotten from the conic
\begin{equation} \label{eq:special_conic}
  \C_{m,n} : mXZ = mnY^2.
\end{equation}
Thus we can say that every $\C$ behaves like one of the $\C_{m,n}$ for the purposes of Diophantine approximation, whether or not it is integrally isomorphic to it.

The $n$-Markoff spectrum appeared in the work of Schmidt \cite[p.~15]{SchMinII} by generalizing the work of Markoff on minima of binary quadratic forms. It is also studied by Vulakh \cite[Section 7]{VulakhMarkov}; see also the survey article of Malyshev (\cite[Section 5]{MalyshevR}; see \cite{Malyshev1981MarkovAL} for English translation). The discrete parts of the $2$- and $5$-Markoff spectra, which were described by Schmidt, have additional combinatorial structure investigated by Abe, Aitchison, and Rittaud \cite{AbeRittaudPalindromes,AbeTwoColor} The $n$-Lagrange spectrum appears to be unstudied in general. It is hoped that this paper rekindles interest in $n$-spectra, which are thus seen to be natural from multiple perspectives. On the other hand, it seems unlikely that the $n$-spectra exhaust the possibilities of Diophantine approximation even on rational curves.

\section{Well-definedness of approximability}

We first resolve a notational gap in our definitions.

\begin{lem}\label{lem:well-def} As the notation suggests, $\lambda\([\Xi]\)$ depends only on $\Xi$, not on $\Sigma$.
\end{lem}
\begin{proof}
  Given $\Xi$, the value of $\Sigma$ is unique up to scaling and adding real multiples of $\Xi$. Scaling $\Sigma$ clearly does not affect the ratio $\sqrt{\size{\Phi(\Sigma)}}/\size{\<\Sigma, P\>}$. Under a transformation $\Sigma \to \Sigma + t \Xi$, we have
  \[
  \frac{\sqrt{\size{\Phi(\Sigma)}}}{\<\Sigma, P\>} \mapsto \frac{\sqrt{\size{\Phi(\Sigma + t\Xi)}}}{\<\Sigma + t\Xi, P\>} = \frac{\sqrt{\size{\Phi(\Sigma) + 2t\<\Sigma, \Xi\> + \Phi(\Xi)}}}{\<\Sigma, P\> + t\<\Xi, P\>} = \frac{\sqrt{\size{\Phi(\Sigma)}}}{\<\Sigma, P\> + t\<\Xi, P\>}.
  \]
  As $[P]$ approaches $[\Xi]$, the ratio
  \[
  \frac{\<\Xi, P\>}{\<\Sigma, P\>}
  \]
  of the two relevant terms in the denominator tends to $0$, since $\<\Xi, -\>$ has a double root along the line $\RR\Xi$ while $\<\Sigma, -\>$ has a simple root.
\end{proof}

\section{Standard forms of conics}
If a conic has a rational point, we may apply an $\SL_3(\ZZ)$-transformation to assume that this point is $[0:0:1]$ and also that the tangent line there is $X = 0$. Thus we need only consider forms of the shape
\begin{equation}\label{eq:CABCD}
  \Phi = \Phi_{A, B, C, D} : (X, Y, Z) \mapsto -A X^2 - B X Y - C Y^2 + D X Z = 0.
\end{equation}
We may also take $C > 0$, $D > 0$. Here the discriminant is $\Delta_\Phi = CD^2$. In particular, if $\Delta_\Phi$ is squarefree, then $D = 1$ and it is easy to see that there is only one isomorphism type of conic. However, in general, the enumeration of conics of given discriminant is a natural and nontrivial problem in its own right. Most of the existing literature concerns definite forms, that is, conics with no $\RR$-points; we believe that the indefinite case is equally worthy of study.

By a classical parametrization (see Gross and Lucianovic \cite[Proposition 4.1]{cubquat}), ternary quadratic forms $\Phi$ over $\ZZ$ up to $\GL_3(\ZZ)$-equivalence (with the $\GL_3$-action twisted by the inverse determinant) parametrize isomorphism classes of quaternion rings over $\ZZ$. Here, since our conics are equivalent over $\QQ$ to $\Phi_1(X,Y,Z) = -Y^2 + X Z$, all of our quaternion rings are lattices $\R \subset \Mat^{2\times 2}(\QQ)$ of full rank. It is easy to see that we may assume $\R \subseteq \Mat^{2\times 2}(\ZZ)$; in fact, following through the construction in \cite{cubquat} shows that a matrix presentation of the ring corresponding to $\Phi = \Phi_{A, B, C, D}$ is
\begin{equation}\label{eq:RABCD}
  \R = \R_{A,B,C,D} = \ZZ\<\mat 1001, \mat{B}{C}{-A}{0}, \mat 000D, \mat 00D0\>.
\end{equation}
Observe that $\Delta_\Phi = CD^2$ reappears as the index $[\Mat^{2\times 2}(\ZZ) : \R]$.

Since $\Mat^{2\times 2}(\QQ)$ has no outer automorphisms as a $\QQ$-algebra, we have:
\begin{prop}
  Two conics $\C_{A, B, C, D}$, $\C_{A', B', C', D'}$ are isomorphic through a transformation in $\SL_3(\ZZ)$ if and only if the corresponding quaternion rings $\R_{A,B,C,D}$, $\R_{A',B',C',D'}$ are conjugate by an element $\gamma \in \GL_2(\QQ)$.
\end{prop}
The link between a conic and its associated matrix ring can be made more concrete:
\begin{lem}
  There is a canonical isomorphism, defined over $\QQ$, between the conic $\C$ and the projectivization $\PP(\V_\C)$ of the standard $2$-dimensional representation of its associated quaternion space $\R_\QQ = \R_\Phi \tensor_\ZZ \QQ \isom \Mat^{2\times 2}(\QQ)$.
\end{lem}
\begin{proof}
  We use the construction of $\R$ as the even part $\E^{\mathrm{even}}$ of the $8$-dimensional Clifford algebra $\E = \Cl(\C, \QQ)$ whose first graded piece $\E^1$ is the $3$-dimensional space on which $\C$ is defined. Let $P \in \C(\QQ)$. Viewing $P$ as an element of $\E^1$, consider
  \[
  \WW = \E^{\mathrm{odd}} \cdot P \subseteq \E^{\mathrm{even}}.
  \]
  If $(P, Q, R)$ is a basis for $\E^{1}$, then $(1, P, Q, R, PQ, QR, PR, QRP)$ is a basis for $\E$. Observe that $P \cdot P = QRP \cdot P = 0$, but $S \cdot P \neq 0$ for any $S \in \E^{1}$ not a multiple of $P$. Thus $\dim \WW = 2$. But $\WW$ is invariant under left-multiplication by $\E^{\mathrm{even}} \isom \R$, so viewing $\R_\QQ$ as $\Mat^{2\times 2} \QQ$, there must exist a nonzero vector $v \in \V_\C$ such that
  \[
  \WW = \left\{M \in \Mat^{2\times 2} \QQ : M v = 0 \right\}.
  \]
  This $v$ is unique up to scaling. It is a short calculation to show that the passage from $P$ to $v$ actually has degree $1$ and thus defines an isomorphism.
\end{proof}
In the case that $\Phi$ is placed in the form \eqref{eq:CABCD} and $\R$ is placed in the form \eqref{eq:RABCD}, we can be more specific: the isomorphism is
\begin{equation}
  \begin{aligned}
   \rho_{\Phi} : \PP(\V_\C) &\rightarrow \C \\
    \vect S T
    &\mapsto [DS^2 : DST : AS^2 + BST + CT^2].
  \end{aligned}
\end{equation}
In other words, it arises from composing the Veronese parametrization
\[
\begin{aligned}
  \rho_{\Phi_1} : \PP^1(\QQ) &\rightarrow \C_1 \\
  \vect S T
  &\mapsto [S^2 : ST : T^2]
\end{aligned}
\]
(where $\Phi_1 = -Y^2 + XZ$ and $\C_1$ is its associated conic) with a self-map of $\PP^2$, the projectivization of
\[
\begin{aligned}
  \tau : \RR^3 &\rightarrow \RR^3 \\
  (x,y,z)
  &\mapsto (D x, D y, A x + B y + C z),
\end{aligned}
\]
which sends $\C_1$ to $\C$. More precisely, we have the relation
\begin{equation}
  \label{eq:tau_qdrc_form}
  \Phi\(\tau(\Sigma)\) = \Delta_\Phi \cdot \Phi_1\(\Sigma\).
\end{equation}

\section{Invariant lattices}
The presentations of $\R_\Phi$ as a subring of $\Mat^{2\times 2}(\ZZ)$ are now parametrized by the $\R_\Phi$-invariant lattices in $\V_\C$, up to scaling. There are finitely many, and we can scale them to be primitive in $\ZZ^2$, necessarily of index dividing $CD$ since $\R_\Phi \supset CD\Mat^{2\times 2}(\ZZ)$. For example, the conic $\C_{0,0,n,1}$ of discriminant $n$ has associated quaternion ring
\[
  \R_{0,0,n,1} = \left\{\mat abcd \in \Mat^{2\times 2}(\ZZ) : b \equiv 0 \mod n\right\}
\]
and invariant lattices
\[
  \Lambda_d = \ZZ\<\vect{d}{0}, \vect{0}{1}\>,
\]
one for each divisor $d|n$.

We now relate approximability to the invariant lattices.
\begin{lem}\label{lem:to_R}
  Let $\Phi$ be a form as in Theorem \ref{thm:main}. Let $\Lambda_1, \ldots, \Lambda_r$ be the $\R_\Phi$-invariant lattices of the associated representation $\V_\C$.  Choose an isomorphism of each $\Lambda_i$ with $\ZZ^2$, which induces an isomorphism $\C \isom \PP^{1}(\QQ)$. Then the approximabilities of points and point pairs can be expressed in terms of approximabilities of real numbers in the classical sense:
  \begin{enumerate}[$($a$)$]
    \item\label{to_R:Markoff} If $[\Xi], [\Xi']$ are points on $\C$, let $\xi_i$, $\xi_i'$ be their coordinates. We have
    \[
      \mu(\Xi,\Xi') = \max_i \mu(\xi_i, \xi_i'),
    \]
    where the right-hand $\mu$'s are defined in \eqref{eq:mu_xi_xi'}.
    \item\label{to_R:Lagrange} If $[\Xi]$ is a point on $\C$, let $\xi_i \in \RR \union \{\infty\}$ be its coordinate in the $i$th parametrization. We have
    \[
    \lambda(\Xi) = \max_i \lambda(\xi_i),
    \]
    where the right-hand $\lambda$'s are defined in \eqref{eq:lambda_xi}
  \end{enumerate}
  In particular, if $\V_\C$ has only one invariant lattice, then $\M(\C) = \M$ and $\L(\C) = \L$ are the classical Markoff and Lagrange spectra.
\end{lem}
\begin{proof}
Place $\Phi$ in the standard form \eqref{eq:CABCD}. In particular, $\R$ when written in the form \eqref{eq:RABCD} fixes a lattice $\Lambda = \ZZ^2$ which we may take to be $\Lambda_1$, and
\[
  [\Xi] = [D\xi^2 : D\xi : A\xi^2 + B\xi + C]
\]
corresponds to the line $\QQ\<\smallvect{\xi}{1}\>$. Let
\[
  [P] = [Ds^2 : Dst : As^2 + Bst + Ct^2]
\]
be a rational point on $\C$, and let $[P_1] = \rho_{\Phi_1}\([Q_1]\)$ and $[\Xi_1] = [\xi^2 : \xi : 1]$ be the corresponding points on the Veronese conic $\C_1$, where $Q_1 = (s,t)$. We have $[P] = \tau\([P_1]\)$, but the actual coordinates in lowest terms
\[
  P = \frac{1}{g} \tau\(P_1\)
\]
differ by the common factor
\[
  g = \gcd(Ds^2, Dst, As^2 + Bst + Ct^2).
\]
A simple calculation shows that
\[
g = \gcd_{M \in \R}(Q_1 \wedge M Q_1),
\]
the wedge product being viewed as a multiple of either generator of $\wedge^2 \Lambda_1$. But $Q_1$ is a primitive vector in $\Lambda_1$ and in $\R Q_1$, so
\[
g = [\Lambda_1 : \R Q_1].
\]
Observe that $\R Q_1$ is also an $\R$-invariant lattice so $\R Q_1 = \Lambda_j$ for some $j$, the $\Lambda_i$ being scaled to be primitive in $\ZZ^2$. Also, let $g_i = [\Lambda_1 : \Lambda_i]$, so that $g = g_j$.

Let $\phi_i : \V_\C \to \V_\C$ be a linear transformation of determinant $g_i$ sending $\Lambda_1$ to $\Lambda_i$. Denote by $\Sym^{2}\phi_i$ its symmetric square, a self-map of $\QQ^3$ and thus of $\PP^2$ preserving $\C_{1}$ and having determinant $g_i^3$. We now have the commutative diagram
\begin{gather*}
\xymatrix@!0@C=2cm@R=1.5cm{
  \V_\C \ar[r]^{\phi_i} \ar[d]^{\rho_{\Phi_1}} & \V_\C \ar[d]^{\rho_{\Phi_1}} \ar[rd]^{\rho_C} \\
  \QQ^3 \ar[r]_{\Sym^{2}\phi_i} & \QQ^3 \ar[r]_{\tau_{\C}} & \QQ^3 
} \\
\xymatrix@!0@=2cm{
\Delta_\Phi g^2 \Phi_1  & \Delta_\Phi \Phi_1 \ar@{|->}[l] & \Phi \ar@{|->}[l].
}\phantom{\Delta g}
\end{gather*}
Then $\xi_i$ is the coordinate in the upper left $\PP(\V_\C)$ corresponding to $\Xi \in \C$, and so forth. The approximating point $P_1 = \rho_{\Phi_1}(Q_1)$ maps back to a point $P_i = \rho_{\Phi_1}(Q_i)$, where $[Q_i] = \phi_i^{-1}\([Q_1]\)$, but the proper scaling is $Q_i = h \cdot \phi_i^{-1}\(Q_1\)$ where
\[
  h = \min \{n > 0 : n Q_1 \in \Lambda_i\}.
\]
By commutativity, $P_1 = h^{-2} \cdot \Sym^2 \phi_i(P_i)$. Let $\Sigma_i = \(\Sym^2\phi_i\)^{-1}(\Sigma)$. We can now compute, using the way $\C$ and its bilinear form transform under the various linear maps,
\[
  \frac{\sqrt{\Delta_\Phi \Size{\Phi(\Sigma)}}}{\Size{\<\Sigma, P\>_\Phi}}
  = \frac{\sqrt{\Delta_\Phi^2 \Size{\Phi_1(\Sigma_1)}}}{\Delta_\Phi \Size{\<\Sigma_1, \frac{1}{g_j}P_1\>_{\Phi_1}}}
  = g_j \frac{\sqrt{\size{\Phi_{1}(\Sigma_1)}}}{\Size{\<\Sigma_1, P_1\>_{\Phi_1}}}
  = g_j \cdot \frac{\sqrt{g_i^2 \size{\Phi_{1}(\Sigma_i)}}}{\Size{g_i^2\<\Sigma_i, \frac{1}{h^2} P_i\>_{\Phi_1}}}
  = \frac{g_j h^2}{g_i} \frac{\sqrt{\size{\Phi_{1}(\Sigma_i)}}}{\Size{\<\Sigma_i, P_i\>_{\Phi_1}}}.
\]
In the last expression, the fraction \[
  \frac{\sqrt{\size{\Phi_{1}(\Sigma_i)}}}{\Size{\<\Sigma_i, P_i\>_{\Phi_1}}}
\]
is the quantity appearing in \eqref{eq:mu_Sigma} and \eqref{eq:lambda_Xi} of $P_i$ as an approximation to $(\Xi_i, \Xi_i')$, respectively $\Xi_i$. By the discussion in Section \ref{sec:prev}, these determine the approximability of $(\xi_i, \xi_i')$, respectively $\xi_i$. The coefficient
\[
  \frac{g_j h^2}{g_i} = [\Lambda_i : h \Lambda_j]
\]
is at least $1$ since $h\Lambda_j = h\R Q_1 \subseteq \Lambda_i$, and equality holds when $i = j$. Hence
\begin{equation} \label{eq:main_lemma}
  \frac{\sqrt{\Delta_\Phi \Size{\Phi(\Sigma)}}}{\Size{\<\Sigma, P\>_\Phi}} = \max_i \frac{\sqrt{\size{\Phi_{1}(\Sigma_i)}}}{\Size{\<\Sigma_i, P_i\>_{\Phi_1}}}.
\end{equation}
Now the proofs for the Markoff and Lagrange spectra finally separate. For the Markoff spectrum, we take the supremum of \eqref{eq:main_lemma} over all rational points $P$ and $P_i$ and get \ref{to_R:Markoff}. For the Lagrange spectrum, we let $P$ tend to $\Xi$, so $P_i$ must tend to $\Xi_i$, and take the limit supremum of \eqref{eq:main_lemma} to get \ref{to_R:Lagrange}.
\end{proof}

At once we get some corollaries.
\begin{cor} \label{cor:immed}
  \begin{enumerate}[$($a$)$]
    \item\label{immed:R} The Lagrange and Markoff spectra of any conic are contained respectively in those of $\RR$. In particular, the $\alpha$-values less than $3$ (if any) form a discrete set and are only attained at points defined over a quadratic field.
    \item\label{immed:same} An isomorphism $\V_\C \isom \V_{\C'}$ of $\QQ$-vector spaces, preserving the set of invariant lattices, induces an isomorphism $\C \isom \C'$ of the corresponding conics that preserves the approximability of every real point.
  \end{enumerate}
\end{cor}
For example, two conics belonging to the same \emph{genus,} that is, that are locally isomorphic, have the same invariant lattices and thus the same Lagrange and Markoff spectra. For example, if $p \equiv 1 \mod 4$ is a prime and $a$ is a quadratic non-residue modulo $p$, one can show that the conics $\C_{1, 0, p, p}$ and $\C_{a^2, 0, p, p}$ of discriminant $p^2$ belong to the same genus but are not isomorphic. However, Corollary \ref{cor:immed}\ref{immed:same} is even more generally applicable. For instance, let $p$ be an odd prime and $a$ a quadratic non-residue modulo $p$. Taking $\C = \C_{-a,0,1,p}$ in \eqref{eq:RABCD} yields a quaternion ring whose only invariant lattice is $\ZZ^2$ because the second generator in \eqref{eq:RABCD} has no eigenvectors defined over $\FF_p$. Hence approximation theory on $\C$ is the same as on $\C_1$, that is to say, on $\RR$, although $\C$ and $\C_1$ are not isomorphic at $p$ nor even have the same discriminant. Thus the Lagrange and Markoff spectra are very far from being a complete invariant for isomorphism types of conics.

\section{Classification of invariant lattice configurations}
\begin{thm}
Let $\C$ be a conic over $\ZZ$ with rational points, and let $\V_\C$ be its associated $2$-dimensional representation. Then there exist positive integers $n$, $m$, with $n m^2 | \Delta_\Phi$ and a basis of $\V_\C$ such that the family of $\R_\Phi$-invariant lattices has the following simple description: a primitive lattice $\Lambda$ is invariant if and only if it has index dividing $m$ in one of the lattices
\[
\Lambda_d = \ZZ\<\vect{d}{0}, \vect{0}{1}\>, \quad d \mid n.
\]
\end{thm}
\begin{proof}
We first fix a prime $p$ and prove the corresponding theorem over $\ZZ_p$. Let $\R_p$ be the completion of $\R = \R_\Phi$ at $p$. Recall that the Bruhat-Tits tree $\T$ at $p$ is an infinite $(p+1)$-regular tree whose vertices are the lattices in $\ZZ_p^2$, up to scaling, and whose edges are the pairs where one lattice has index $p$ in the other. Let $\G$ be the induced subgraph of $\T$ formed by the invariant lattices. The relevant local statement is as follows:
\begin{lem}\label{lem:nbhd}
$\G$ is the $k$-neighborhood of a path in $\T$, where the \emph{$k$-neighborhood} of a subset $\S$ consists of the vertices at most distance $k$ from $\S$. Moreover, the number of $\ell$ of edges in the path is bounded by
\begin{equation} \label{eq:diam}
  2k + \ell \leq v_p(\Delta_\Phi).
\end{equation}
\end{lem}
\begin{proof}
Observe that $\G$ has the following features:
\begin{enumerate}
  \item $\G$ is connected. Indeed, for any vertices $\Lambda_1$, $\Lambda_2$ of $\G$, the path from $\Lambda_1$ to $\Lambda_2$ is made from sums $p^{i}\Lambda_1 + p^j\Lambda_2$, which are also $\R_p$-invariant.
  \item $\G$ is finite. Indeed, the diameter of $\G$ is at most $v_p(\Delta_\Phi)$ since $\Delta_\Phi \Mat^{2\times 2}(\ZZ) \subseteq \R_\Phi$.
  \item Each vertex $\Lambda$ of $\G$ has degree $0$, $1$, $2$, or $p+1$. Indeed, if we write $\R_p$ as a matrix ring using a basis of $\Lambda$, the neighbors of $\Lambda$ are the lattices $\<v\> + p\Lambda$, where $v$ is a simultaneous eigenvector of the reductions of elements of $\R_p$ mod $p$. If $\R_p$ contains a non-scalar matrix mod $p$, then this matrix has at most two eigenvectors, up to scaling; otherwise, any $v$ will work.
\end{enumerate}
We now proceed by induction on the number of vertices of $\G$. If there are no vertices of degree $p+1$, then $\G$ is a path. Otherwise, assume that $\ZZ_p^2$ is a vertex of degree $p+1$. Then $\R_p = \ZZ_p + p\breve{\R}$, where $\breve{\R} = \frac{1}{p}\R_p \intsec \Mat^{2\times 2} \ZZ_p$ is a lattice that need not be a ring. Let $\R'$ be the ring generated by $\breve{\R}$. Observe that $\R'$ is a lattice of finite index in $\Mat^{2\times 2} \ZZ_p$ and so corresponds to some conic over $\ZZ_p$.

We claim that a primitive lattice $\Lambda \subsetneq \ZZ_p^2$ is $\R_p$-invariant if and only if its neighbor in the direction of $\ZZ_p^2$,
\[
  \Lambda' = \frac{1}{p} \Lambda \intsec \ZZ_p^2,
\]
is $\R'$-invariant. For the forward direction, it suffices to show that $\Lambda'$ is $\breve{\R}$-invariant. Let $p^k = [\ZZ_p^2 : \Lambda]$, and observe that
\[
  \Lambda' = \Lambda + p^{k-1}\ZZ_p^2
\]
so
\begin{align*}
  \breve{\R} \Lambda' &= \(\frac{1}{p} \R_p \intsec \Mat^{2\times 2}\ZZ_p\)\(\Lambda + p^{k-1} \ZZ_p^2 \) \\
  &\subseteq \(\frac{1}{p} \R_p \Lambda \intsec \ZZ_p^2 \) + p^{k-1} \ZZ_p^2 \\
  &\subseteq \(\frac{1}{p} \Lambda \intsec \ZZ_p^2\) + p^{k-1} \ZZ_p^2 \\
  &= \Lambda'. 
\end{align*}
Conversely, if $\R'\Lambda' \subseteq \Lambda'$, then $\R_p\Lambda \subseteq (\ZZ_p + p\R')\Lambda \subseteq \Lambda + p\R'\Lambda' \subseteq \Lambda + p\Lambda' = \Lambda$.

Let $\G'$ be the graph of invariant lattices of $\R'$. We have now shown that $\G$ is the $1$-neighborhood of $\G'$. Since $\G'$ has fewer vertices than $\G$, we can apply the induction hypothesis to deduce that $\G'$ is a $k$-neighborhood of a path. Hence $\G$ is the $(k+1)$-neighborhood of this path, completing the induction.

Since the $k$-neighborhood of a path of length $\ell$ has diameter $2k + \ell$, we also have \eqref{eq:diam}.
\end{proof}
Now any path in $\T$ can be viewed under a suitable basis as the family of lattices
\[
  \Lambda_{i, p} = \ZZ_p\<\vect{p^{i}}{0}, \vect{0}{1}\>, \quad 0 \leq i \leq \ell.
\]
The basis is far from unique; it is only necessary that the two basis vectors generate $\Lambda_{0, p}$, the lattice at an endpoint of the path, and that the second basis vector lies on the correct line modulo $p^\ell$.

We now vary $p$. Let $\Lambda_0$ be a lattice whose completions are the appropriate $\Lambda_{0, p}$ at each $p$, and we aggregate the parameters $k = k_p$ and $\ell = \ell_p$ to produce the invariants
\[
  m = \prod_p p^{k_p}, \quad n = \prod_p p^{\ell_p},
\]
with $m^2 n | \Delta_\Phi$. We then take a basis for $\Lambda_0$ whose second vector satisfies the needed congruences at the prime powers dividing $\Delta_\Phi$ to finish the proof.
\end{proof}

We can now finish the main theorem. 
\begin{proof}[Proof of Theorem \ref{thm:main}]
Let the parametrization of $\C$ be chosen so that the invariant lattices have the form of all lattices of index $m$ in $\Lambda_d$ for $d|n$. Assume first that $m = 1$. Then in Lemma \ref{lem:to_R}, the $\xi_i$ are the coordinates of a single ray in each of the $\Lambda_d$, which are related by $\xi_d = d\xi$. So the approximability of $\Xi$ is the maximum of those of the $d\xi$. We can also derive the alternative formulas for the approximability in Theorem \ref{thm:main} as follows. When expressing this in terms of rational approximations to $\xi$, a single value $s/t$ is seen to give the highest quality of $ds/t$ when $d$ is chosen to minimize the numerator and denominator, namely $d = \gcd(t, n)$, as desired.

When $m > 1$, we do not proceed directly. Instead, we note that $\C$ is isomorphic over $\QQ$ to the special conic $\C_{m,n} : m XZ = mn Y^2$, via a map that may not be integral over $\ZZ$ but preserves all the invariant lattices. By Corollary \ref{cor:immed}\ref{immed:same}, this map preserves approximabilities of points. Now the factor of $m$ when passing from $\C_{1,n}$ to $\C_{m,n}$ is seen, directly from the definition, to scale approximabilities by a factor of $m$.
\end{proof}

The multiplier $m$ acts on the Lagrange and Markoff spectra in a predictable way, scaling all approximabilities by $m$. This happens most clearly when the equation for $\C$ is multiplied through by $m$, but it can also happen that $m > 1$ for primitive conics. For instance, the invariant lattices of the conic $\C_{0,0,1,p^2} : p^2 XZ = Y^2$ consist of one lattice and all its sublattices of index $p$, the same as for $\C_1$ under the imprimitive form $-pY^2 + p XZ$. Accordingly, we can focus on the Markoff spectrum of $\C_{1,n}$, which Schmidt calls the $n$-spectrum in \cite[p.~15]{SchMinII}, and its Lagrange analogue.

\section{Unanswered questions}

In summary, we have shown that the Lagrange and Markoff spectra $\L(\C)$, $\M(\C)$ of any conic, after an integer scale factor $m$, reduce to a one-parameter family of \emph{$n$-spectra} $\L_n$, $\M_n$. The proof of Lemma \ref{lem:nbhd} is not very explicit, and it is desirable to have a formula for $m$ and $n$ in terms of the coefficients $A$, $B$, $C$, $D$.

It is hoped that this paper rekindles interest in the $n$-spectra. With a bit more work, one can likely replicate the proofs of classical results for the $n = 1$ case compiled by Cusick and Flahive \cite{CusickFlahive}:
\begin{enumerate}
  \item $\L_n$ is the closure of the set of approximabilities of quadratic irrationals \cite[Theorem 3.2]{CusickFlahive}. Likewise, $\M_n$ is the closure of the set of approximabilities of pairs of quadratic irrationals (not necessarily belonging to the same quadratic field) \cite[Theorem 3.3]{CusickFlahive}.
  \item Each $\L_n$ and $\M_n$ subdivides into an initial discrete segment, a totally disconnected segment, and a Hall's ray.
  \item The initial discrete parts of $\L_n$ and $\M_n$ coincide and are given by quadratic irrationals. Tools for understanding the initial discrete part of $\M_n$ for general $n$ have been developed by Vulakh \cite{VulakhMarkov}.
\end{enumerate}
However, the presence of the varying parameter $n$ allows for some deeper questions, for instance:
\begin{enumerate}[resume]
  \item What is the least element of $\L_n$, resp{.} $\M_n$? The smallest possible value $\sqrt{5}$ occurs for infinitely many $n$, namely those having no prime factor congruent to $2$ or $3$ modulo $5$, and with $5^2 \nmid n$. On the other hand, it is apparently possible that none of the infinitely many Markoff numbers that make up $\M_1 \intsec (0,3)$ appears in $\M_{n}$ (see the $n = 12$ case of Table \ref{tab:Mn}).
  \item What is the least limit point of $\L_n$, resp{.} $\M_n$? The data in this paper suggest that it is either an integer or a quadratic irrationality defined over the same quadratic field as $\min(\M_n)$.
  \item Is there a uniform estimate for the endpoint of Hall's ray, as $n$ varies?
  \item Is $\L_n \neq \M_n$ for all $n$?
\end{enumerate}

\appendix
\section{Examples}
We close with two tables to shed light on potential research directions. In Table \ref{tab:local}, we list conics over $\ZZ_p$ of low discriminant, together with the $k$ and $\ell$ (path length) of Lemma \ref{lem:nbhd}, which are the $p$-adic valuations of the $m$ and $n$, respectively, of Theorem \ref{thm:main}. The proof of Lemma \ref{lem:nbhd} is not entirely explicit, and one could hope for a straightforward formula for these invariants in terms of $A$, $B$, $C$, and $D$. (For brevity, the case $p = 2$ is not listed; for clearer patterns, one might better consider conics over a general extension of $\QQ_2$.)

In Table \ref{tab:Mn}, we list low-lying approximabilities lying in the $n$-Lagrange and hence the $n$-Markoff spectrum. The coincidence with the known cases shows that even a na\"ive algorithm, based on trying all continued fractions of bounded length and terms, is enough to recover the lowest points and first limit point of the spectrum when these are known.

Note that, except in the cases $n = 1,2,5,6$ where a Markoff-like tree structure obtains, the first few approximability values arise from points $\xi$ whose associated continued fraction expansions have the form
\[
  A, B, A^2 B, A^4 B, A^6 B, \ldots
\]
where $A$ and $B$ are words, up to the first limit point. Examination of subsequent values (not shown in the table) reveals a complementary pattern,
\[
  \ldots, A^5 B, A^3 B, AB,
\]
raising the question of whether it is possible to find and describe discrete portions of the spectrum \emph{after} the first limit point.

\newpage

\begin{longtable}{r|lcc}
$v_p(\Delta_\Phi)$ & $\C$ & $k$ & $\ell$ \\ \hline
\endhead
$0$\tall & $Y^2 = XZ$ & $0$ & $0$ \\ \hline
$1$\tall & $pY^2 = XZ$ & $0$ & $1$ \\ \hline
$2$\tall & $p^2 Y^2 = XZ \sim -X^2 + Y^2 = pXZ$ & $0$ & $2$ \\
& $-aX^2 + Y^2 = pXZ$ & $0$ & $0$ \\
& $Y^2 = pXZ$ & $0$ & $1$ \\ \hline
$3$\tall & $p^3 Y^2 = XZ \sim X Y + pY^2 = pX Z$ & $0$ & $3$ \\
& $-X^2 + pY^2 = pXZ$ & $0$ & $1$ \\
& $-a X^2 + pY^2 = pXZ$ & $0$ & $1$ \\
& $pY^2 = pXZ$ (imprimitive) & $1$ & $0$ \\ \hline
$4$\tall & $p^4 Y^2 = XZ \sim XY + p^2 Y^2 = p XZ \sim -X^2 + Y^2 = p^2 XZ$ & $0$ & $4$ \\
& $-aX^2 + Y^2 = p^2 XZ$ & $0$ & $0$ \\
& $Y^2 = p^2 XZ \sim -X^2 + p^2 Y^2 = pXZ $ & $1$ & $0$ \\
& $-aX^2 + p^2 Y^2 = p XZ $ & $1$ & $0$ \\
& $-pX^2 + p^2 Y^2 = p XZ $ & $1$ & $1$ \\
& $-paX^2 + p^2 Y^2 = p XZ $ & $1$ & $1$ \\
& $p^2 Y^2 = p XZ $ (imprimitive) & $1$ & $1$ \\
\caption{Contribution to the multiplier $m = \prod p^k$ and the conductor $n = \prod p^\ell$ of the spectra of conics with various completions at an odd prime $p$. Here $a$ denotes a quadratic non-residue modulo $p$, and $\sim$ denotes equivalence under $\GL_3(\ZZ_p)$.}
\label{tab:local}
\end{longtable}

\begin{longtable}{r|lcl}
  $n$ & First few approximabilities with associated continued fractions &%
  \begin{tabular}[b]{l}%
    First known\\%
    limit point
  \end{tabular} & Reference \\ \hline
  \endhead
  $1$\tall & 
  \begin{tabular}[t]{cl}
    $\alpha$ & $\xi$ \\ \hline\tall
    $2.2360\ldots$ & $[\bar{1}] \sim (1 + \sqrt{5})/2$ \\
    $2.8284\ldots$ & $[\bar{2}]$ \\
    $2.9732\ldots$ & $[\ba{1122}]$ \\
    $2.9960\ldots$ & $[\ba{111122}]$ \\
    $2.9992\ldots$ & $[\ba{112222}]$
  \end{tabular}
  & $3$ & \cite{MarkoffI,MarkoffII,SchMinI}\hspace{-2pt} \\
  $2$\tall & 
  \begin{tabular}[t]{cll}
    $\alpha$ & $\xi$ & $2\xi$ \\ \hline\tall
    $2.8284\ldots$ & $[\bar{2}] \sim \sqrt{2}             $ & $[\bar{2}]             $ \\
    $3.4641\ldots$ & $[\ba{12}]            $ & $[\ba{12}]            $ \\
    $3.8873\ldots$ & $[\ba{111133}]    $ & $[\ba{111133}]    $ \\
    $3.9799\ldots$ & $[\ba{1123}]        $ & $[\ba{1132}]        $ \\
    $3.9994\ldots$ & $[\ba{112223}]    $ & $[\ba{113222}]    $ \\
    $3.9998\ldots$ & $[\ba{11112133}]$ & $[\ba{11113312}]$ \\
    $3.9999\ldots$ & $[\ba{11222223}]$ & $[\ba{11322222}]$
  \end{tabular}
  & $4$ & \cite{SchMinI,ChaIntrinsic} \\
  $3$\tall & 
  \begin{tabular}[t]{cll}
    $\alpha$ & $\xi$ & $3\xi$ \\ \hline\tall
    $3.4641\ldots$ & $[\ba{12}] \sim \sqrt{3}                   $ & $[\ba{12}]                   $ \\
    $3.6055\ldots$ & $[\bar{3}]                     $ & $[\bar{3}]                     $ \\
    $3.9949\ldots$ & $[\ba{1233}]             $ & $[\ba{1332}]             $ \\
    $3.9999\ldots$ & $[\ba{123333}])      $ & $[\ba{133332}])      $ \\
    $3.9999\ldots$ & $[\ba{12333333}]]$ & $[\ba{13333332}]]$ \\
  \end{tabular}
  & $4$ & \cite{SchMinII, VulakhMarkov, ChaEis}\hspace{-2pt} \\
  $4$\tall &
  \begin{tabular}[t]{clll}
  $\alpha$ & $\xi$ & $2\xi$ & $4\xi$ \\ \hline\tall
    $4.1231\ldots$ & $[\ba{113}] \sim \dfrac{1 + \sqrt{17}\tall}{2}$ & $[\ba{113}]$ & $[\ba{113}]$ \\
    $4.2163\ldots$ & $[\ba{1323}]$ & $[\ba{112}]$ & $[\ba{1323}]$ \\
    $4.3362\ldots$ & $[\ba{1131131323}]$ & $[\ba{112112113113}]$ & $[\ba{1131132313}]$ \\
    $4.3362\ldots$ & $[\ba{(113)^4 1323}]$ & $[\ba{112112(113)^4}]$ & $[\ba{(113)^4 2313}]$ \\
  \end{tabular}
  & \multicolumn{2}{l}{$\dfrac{177 + 37 \sqrt{17}}{76}$} \\
  $5$\tall &
  \begin{tabular}[t]{cll}
    $\alpha$ & $\xi$ & $5\xi$ \\ \hline\tall
    $2.2360\ldots$ & $ [\ba{1}]$ & $[\ba{1}]$            \\
    $4.4721\ldots$ & $ [\ba{4}]$ & $[\ba{4}]$            \\
    $4.5825\ldots$ & $ [\ba{13}]$ & $[\ba{13}]$          \\
    $4.7726\ldots$ & $ [\ba{1144}]$ & $[\ba{1144}]$      \\
    $4.8989\ldots$ & $ [\ba{24}]$ & $[\ba{24}]$          \\
    $4.9590\ldots$ & $ [\ba{114224}]$ & $[\ba{114224}]$  \\
    $4.9752\ldots$ & $ [\ba{133144}]$ & $[\ba{133144}]$  \\
    $4.9839\ldots$ & $  [\ba{1124}]$ & $[\ba{1142}]$     \\
    $4.9976\ldots$ & $  [\ba{111124}]$ & $[\ba{111142}]$ \\
    $4.9986\ldots$ & $  [\ba{1344}]$ & $[\ba{1443}]$     \\
    $4.9997\ldots$ & $ [\ba{113144}]$ & $[\ba{114413}]$  \\
    $4.9999\ldots$ & $ [\ba{114244}]$ & $[\ba{114424}]$  \\
    $4.9999\ldots$ & $  [\ba{134444}]$ & $[\ba{144443}]$
  \end{tabular}
  & $5$ & \cite{SchMinI}\\
  $6$\tall &
  \begin{tabular}[t]{cllll}
    $\alpha$ & $\xi$ & $2\xi$ & $3\xi$ & $6\xi$ \\ \hline\tall
  $3.4641\ldots$ & $[\ba{12}]$ & $[\ba{12}]$ & $[\ba{12}]$ & $[\ba{12}]$ \\
  $4.8989\ldots$ & $[\ba{24}]$ & $[\ba{24}]$ & $[\ba{24}]$ & $[\ba{24}]$ \\
  $5.2915\ldots$ & $[\ba{1114}]$ & $[\ba{1114}]$ & $[\ba{1114}]$ & $[\ba{1114}]$ \\
  $5.7445\ldots$ & $[\ba{1252}]$ & $[\ba{1252}]$ & $[\ba{1252}]$ & $[\ba{1252}]$ \\
  $5.9194\ldots$ & $[\ba{113125}]$ & $[\ba{115213}]$ & $[\ba{115213}]$ & $[\ba{113125}]$ \\
  $5.9254\ldots$ & $[\ba{111532}]$ & $[\ba{111532}]$ & $[\ba{111235}]$ & $[\ba{111235}]$ \\
  \end{tabular}
  & $6$ & \cite{SchMinII} \\
  $7$\tall &
  \begin{tabular}[t]{cll}
    $\alpha$ & $\xi$ & $7\xi$ \\ \hline\tall
  $2.8284\ldots$ & $[\bar{2}]$ & $[\bar{2}]$ \\
  $3.7416\ldots$ & $[\ba{1232}]$ & $[\ba{1232}]$ \\
  $3.7823\ldots$ & $[\ba{122232}]$ & $[\ba{123222}]$ \\
  $3.7835\ldots$ & $[\ba{12222232}]$ & $[\ba{12322222}]$ \\
  $3.7836\ldots$ & $[\ba{1222222232}]$ & $[\ba{1232222222}]$
  \end{tabular}
  & $\dfrac{6\tall}{3 - \sqrt{2}}$ \\
  $8$\tall &
  \begin{tabular}[t]{cllll}
    $\alpha$ & $\xi$ & $2\xi$ & $4\xi$ & $8\xi$ \\ \hline\tall
  $4.1231\ldots$ & $[\ba{113}]$ & $[\ba{113}]$ & $[\ba{113}]$ & $[\ba{113}]$ \\
  $4.8989\ldots$ & $[\ba{1112}]$ & $[\ba{24}]$ & $[\ba{24}]$ & $[\ba{1112}]$ \\
  $4.9999\ldots$ & $[\ba{1113113112}]$ & $[\ba{11311324}]$ & $[\ba{11311423}]$ & $[\ba{1112113113}]$\hspace{-5pt} \\
  $4.9999\ldots$ & $[\ba{1(113)^4112}]$ & $[\ba{(113)^424}]$ & $[\ba{(311)^442}]$ & $[\ba{1112(113)^4}]$ \\
  \end{tabular}
  & $5$ \\
  $9$\tall &
  \begin{tabular}[t]{clll}
    $\alpha$ & $\xi$ & $3\xi$ & $9\xi$ \\ \hline\tall
  $3.6055\ldots$ & $[\bar{3}]$ & $[\bar{3}]$ & $[\bar{3}]$ \\
  $5.2915\ldots$ & $[\ba{1114}]$ & $[\ba{1114}]$ & $[\ba{1114}]$ \\
  $5.3935\ldots$ & $[\ba{111334}]$ & $[\ba{123214}]$ & $[\ba{111433}]$ \\
  $5.3944\ldots$ & $[\ba{11133334}]$ & $[\ba{12333214}]$ & $[\ba{11143333}]$
  \end{tabular}
  & $9 - \sqrt{13}$ \\
  $10$\tall &
  \begin{tabular}[t]{cllll}
    $\alpha$             & $\xi$           & $2\xi$              & $5\xi$          & $10\xi$             \\ \hline
    \tall $4.4721\ldots$ & $[\bar{1}]$         & $[\bar{4}]$     & $[\bar{1}]$         & $[\bar{4}]$     \\
    $4.8989\ldots$       & $[\ba{24}]$         & $[\ba{24}]$     & $[\ba{24}]$         & $[\ba{24}]$     \\
    $5.6180\ldots$       & $[\ba{1111421124}]$ & $[\ba{233255}]$ & $[\ba{1111421124}]$ & $[\ba{233255}]$ \\
    $5.6662\ldots$       & $[\ba{11111124}]$   & $[\ba{2345}]$   & $[\ba{11111142}]$   & $[\ba{2543}]$   \\
    $5.6691\ldots$       & $[\ba{1^{12}24}]$   & $[\ba{234445}]$ & $[\ba{1^{12}42}]$   & $[\ba{254443}]$ 
  \end{tabular}
  & $\dfrac{10\tall}{4 - \sqrt{5}}$
  \\
  $11$\tall &
  \begin{tabular}[t]{cll}
    $\alpha$ & $\xi$ & $11\xi$ \\ \hline\tall
  $2.2360\ldots$ & $[\bar{1}]$ & $[\bar{1}]$ \\
  $3.4641\ldots$ & $[\ba{12}]$ & $[\ba{12}]$ \\
  $3.8729\ldots$ & $[\ba{23}]$ & $[\ba{23}]$ \\
  $3.9799\ldots$ & $[\ba{1123}]$ & $[\ba{1132}]$ \\
  $3.9970\ldots$ & $[\ba{111123}]$ & $[\ba{111132}]$
  \end{tabular}
  & $4$ \\
  $12$\tall &
  \begin{tabular}[t]{cllll}
  $\alpha$ & $\xi, 4\xi$ & $2 \xi$ & $3\xi, 12\xi$ & $6\xi$ \\  
    \hline\tall
  $5.7445\ldots$ & $[\ba{1252}] \sim \frac{3 + \sqrt{33}}{4}$ & $[\ba{1252}]$ & $[\ba{1252}]$ & $[\ba{1252}]$ \\
  $6.0052\ldots$ & $[\ba{12412532}]$ & $[\ba{11152}]$ & $[\ba{12352142}]$ & $[\ba{11125}]$ \\
  $6.0053\ldots$ & \tiny $[\ba{124(1252)^212532}]$ & \tiny $[\ba{11152(1252)^2}]$ & \tiny $[\ba{12352(1252)^2 142}]$ & \tiny $[\ba{11(1252)^2 125}]$\hspace{-10pt}
  \end{tabular}
  & $\dfrac{41 \sqrt{33} - 21}{5 \sqrt{33} + 7}$
  & \\
  $13$\tall &
  \begin{tabular}[t]{cll}
    $\alpha$ & $\xi$ & $13\xi$ \\ \hline\tall
  $2.9732\ldots$ & $[\ba{1122}] \sim \frac{9 + \sqrt{221}}{10}$ & $[\ba{1122}]$ \\
  $3.2811\ldots$ & $[\ba{11212212}]$ & $[\ba{11212212}]$ \\
  $3.2827\ldots$ & $[\ba{(1221)^32112}]$ & $[\ba{1221(2112)^3}]$
  \end{tabular}
  & $\dfrac{195}{104 - 3\sqrt{221}}$
  & \cite{VulakhMarkov} \\ \\
\caption{Some low-lying points in the $n$-Lagrange and $n$-Markoff spectra; conjectures concerning the first limit point; and references for proofs of minimality when known. The symbol $\sim$ denotes equivalence under a linear fractional transformation in $\GL_2\ZZ$. For brevity, the continued fractions are written without term separators (none of the terms exceeds $5$ for these $n$), only the periodic portion of the continued fraction is shown, and the notation $W^k$ indicates that the sequence $W$ of terms is repeated $k$ times. }
\label{tab:Mn}
\end{longtable}

\bibliography{../Master}
\bibliographystyle{amsplain}


\end{document}